\documentclass{article}

\usepackage{amsfonts}
\usepackage{amssymb}
\usepackage{amsthm}
\usepackage{amsmath}
\usepackage{comment}
\usepackage{tikz}
\usetikzlibrary{arrows,decorations.pathmorphing,decorations.pathreplacing,positioning,shapes.geometric,shapes.misc,decorations.markings,decorations.fractals,calc,patterns}
\tikzset{>=stealth',
         cvertex/.style={circle,draw=black,inner sep=1pt,outer sep=3pt},
         vertex/.style={circle,fill=black,inner sep=1pt,outer sep=3pt},
         star/.style={circle,fill=yellow,inner sep=0.75pt,outer sep=0.75pt},
         tvertex/.style={inner sep=1pt,font=\scriptsize},
         gap/.style={inner sep=0.5pt,fill=white}}

\newcommand{\ZZ}{\mathbb{Z}}
\newcommand{\CC}{\mathbb{C}}
\newcommand{\QQ}{\mathbb{Q}}

\newcommand{\FF}{\mathbb{F}}
\newcommand{\GGm}{\mathbb{G}_m}
\newcommand{\FFc}{\bar{\mathbb{F}}}
\newcommand{\CP}{\mathcal{P}}
\newcommand{\GG}{\mathcal{G}}
\newcommand{\OO}{\mathcal{O}}
\newcommand{\DD}{\mathcal{D}}
\newcommand{\EE}{\mathcal{E}}
\newcommand{\HH}{\mathcal{H}}

\newcommand{\CL}{\mathcal{L}}

\newcommand{\som}{\underline{\omega}}
\newcommand{\GZ}{\Gamma_0}
\newcommand{\GI}{\Gamma_1}
\newcommand{\spn}{\mbox{span}}

\title{Generators of graded rings of modular forms}
\author{Nadim Rustom}

\newtheorem{thm}{Theorem}
\newtheorem{rem}{Remark}
\newtheorem{conj}{Conjecture}
\newtheorem{cor}{Corollary}
\newtheorem{lem}{Lemma}
\newtheorem{prop}{Proposition}
\newtheorem{defn}{Definition}
\newtheorem{algo}{Algorithm}

\begin{document}
\maketitle

\begin{abstract}
We study graded rings of modular forms over congruence subgroups, with coefficients in a subring $A$ of $\CC$, and specifically the highest weight needed to generate these rings as $A$-algebras. In particular, we determine upper bounds, independent of $N$, for the highest needed weight that generates the $\CC$-algebras of modular forms over $\Gamma_1(N)$ and $\Gamma_0(N)$ with some conditions on $N$. For $N \geq 5$, we prove that the $\ZZ[1/N]$-algebra of modular forms over $\Gamma_1(N)$ with coefficients in $\ZZ[1/N]$ is generated in weight at most $3$. We give an algorithm that computes the generators, and supply some computations that allow us to state two conjectures concerning the situation over $\Gamma_0(N)$. 
\end{abstract}

\section{Introduction}
Let $A$ be a subring of $\CC$, and $\Gamma$ a congruence subgroup of $SL_2(\ZZ)$. A modular form (on $\Gamma$) is said to have coefficients in $A$ (or to be $A$-integral) if the Fourier coefficients of its expansion at infinity lie in $A$. The space of modular forms with coefficients in $A$, and weight $k$, is denoted by $M_k(\Gamma,A)$, and acquires the structure of an $A$-module. Then we have the graded $A$-algebra:
\[ M(\Gamma,A) = \bigoplus_{k=0}^{\infty} M_k (\Gamma, A). \]
When $A$ is not specified, we take it to mean that $A = \CC$, and we simply write $M(\Gamma)$ for the corresponding graded $\CC$-algebra.\\
It can be shown that $M(\Gamma, \ZZ)$ is finitely generated (\cite{DR72}, Theorem 3.4). In this note, we are interested in questions concerning the generators. First, we explore the situation over $\CC$ for various congruence subgroups. In the case of $\Gamma_1(N)$, Borisov and Gunnells prove in \cite{BG02}, using different methods than the ones we use in this paper (the theory of toric modular forms), that for prime $N$, the $\CC$-algebra $M(\Gamma_1(N),\CC)$ is generated in weight at most $3$. For $\Gamma(N)$, with $N \geq 3$, Khuri-Makdisi proves in \cite{Makdisi12} that weight 1 is enough to generate $M(\Gamma(N),\CC)$, by explicitly exhibiting a subalgebra $\mathcal{R}_N$, generated by weight $1$ Eisenstein series, and containing all modular forms of weight at least $2$. Khuri-Makdisi's proof uses the description of modular forms as global sections of a line bundle on the modular curve, and in particular, a lemma on the surjectivity of multiplication between tensor powers of a line bundle. We employ this lemma to generalize Borisov and Gunnell's result for $\Gamma_1(N)$ for arbitrary $N \geq 3$, and to give a similar result for $\Gamma_0(N)$ with some conditions on $N$. In particular, we prove the following: \\\\
\textbf{Corollary 1. } \textit{The $\CC$-algebra $M(\Gamma,\CC)$ of modular forms is generated:
\begin{enumerate}
\item in weight at most $3$, when $\Gamma = \Gamma_1(N)$ for $N \geq 5$, or
\item in weight at most $6$, when $\Gamma = \Gamma_0(N)$ for $N$ satisfying the conditions of Lemma \ref{GZell}.
\end{enumerate}}
Next, we attempt to generalize the results, using the theory of moduli schemes of elliptic curves, for which we provide a brief summary, and we prove the following theorem:
\begin{thm}
Let $N$ be an integer, with $N \geq 5$. Then $M(\Gamma_1(N),\ZZ[1/N])$ is generated in weight at most 3 as a $\ZZ[1/N]$-algebra.
\end{thm}
By ``generated in weight at most 3" we mean that any $\ZZ[1/N]$-subalgebra of $M(\Gamma_1(N),\ZZ[1/N])$ that contains all modular forms with coefficients in $\ZZ[1/N]$ and of weight less than or equal to 3, must be equal to the whole algebra $M(\Gamma_1(N),\ZZ[1/N])$. 
We then turn our attention to the modular forms over $\Gamma_0(N)$. Using the work of B\"ocherer and Nebe in \cite{BN10}, we show that the algebra $M(\Gamma_0(N))$ is generated in weight at most 10 when $N$ is square-free. Following that, we describe an argument due to Scholl, which provides a blueprint for proving that a certain algebra $M(\Gamma,A)$ is finitely generated, provided it satisfies some conditions, which include the existence of a special modular form, called a $T$-form. We also describe Scholl's construction of a $T$-form for $\Gamma_0(N)$, and make it explicit for the case of prime level. We provide an algorithm that computes the generators based on Scholl's proof. The running time of the algorithm depends on the vanishing order of the $T$-form at infinity. We show that Scholl's $T$-form is in general not optimal in that respect, and supply the optimal $T$-form in the case of prime level. Finally, we provide some numerical data concerning the rings $M(\Gamma_0(N),\ZZ[1/N])$ and $M(\Gamma_0(N),\CC)$, which we use to state two conjectures concerning the maximal needed weight for generators of these rings.\\\\
\textbf{Acknowledgements. }The author wishes to thank Kamal Khuri-Makdisi for his helpful remarks, and for the simplification of the proof of Lemma \ref{linalg}. The author would also like to thank David Loeffler for his insightful and helpful comments. 

\section{Modular forms of level $N$ over $\CC$}

In this section, we consider modular forms of level $N$ with coefficients in $\CC$. We will use the the following lemma from the theory of line bundles on algebraic curves; for a proof, refer to \cite{Makdisi04}.
\begin{lem}(Surjectivity of multiplication)\label{surjmul}
Let $X$ be a smooth, geometrically connected algebraic curve of genus $g$ over a perfect field $k$, and let $\CL_1$ and $\CL_2$ be two line bundles on $X$ (both defined over $k$) of degrees $d_1$ and $d_2$. If $d_1, d_2 \geq 2g + 1$, then the canonical multiplication map:
\[H^0(X,\CL_1) \otimes H^0(X,\CL_2) \rightarrow H^0(X,\CL_1 \otimes \CL_2)\]
is surjective.
\end{lem}
We use this lemma to study the situation over $\Gamma_1(N)$ and $\Gamma_0(N)$. 
Modular forms of a given level may be viewed as holomorphic section of a line bundle, or invertible sheaf, on the corresponding modular curve. This construction is explained in detail in Section 12.1 of \cite{DI95}. Specifically, let $\Gamma$ be a congruence subgroup of $SL_2(\ZZ)$, and $X = \HH^* / \Gamma$ be the corresponding compact modular curve. First, suppose that $-1 \not \in \Gamma$. Then for each positive integer $k$, there exists an invertible sheaf $\GG_k$ such that (12.1.1 in \cite{DI95}):
\[ M_k(\Gamma) = H^0(X,\GG_k) \]
In general, still assuming $-1 \not \in \Gamma$, let $\DD_k$ be the sheaf of holomorphic functions with zeroes of order at least $k/2$ at irregular cusps, and let $\EE_k$ be the sheaf of functions with zeroes of order at least $k/4$ (respectively, $k/3$) at elliptic points of order $2$ (respectively, of order $3$). Then on stalks at these points, we have the isomorphism (12.1.7 in \cite{DI95}):
\[ \GG_1^{\otimes_{\OO_X} k} \cong \GG_k \otimes_{\OO_X} \DD_k \otimes_{\OO_X} \EE_k. \]
If $-1 \in \Gamma$, then for each positive odd $k$, we have $M_k(\Gamma) = 0$. Furthermore, for each positive even integer $k$, we get, as above, an invertible sheaf $\GG_k$, with:
\[ M_k(\Gamma) = H^0(X, \GG_k),\]
and we have the isomorphism:
\[ \GG_2^{\otimes_{\OO_X} (k/2)} \cong \GG_k \otimes_{\OO_X} \EE_k.\]
This shows that, whenever $X$ has no elliptic points and no irregular cusps, there exists a line bundle $\CL$ such that:
\[M_k(\Gamma) = H^0(X, \CL^{\otimes k})\]
for all positive integers $k$, if $-1 \not \in \Gamma$, and
\[M_{2k}(\Gamma) = H^0(X, \CL^{\otimes k}) \]
for all positive integers $k$, if $-1 \in \Gamma$, with $M_k(\Gamma) = 0$ for odd $k$.\\
Assume from this point on that $\Gamma$ has no elliptic points or irregular cusps, so that a line bundle $\CL$ as described above exists. The degree of a line bundle is the degree of (any) one of its global sections. Thus using Proposition 2.16 in $\cite{Shi71}$, and using $k=1$ or $k=2$ depending on whether $-1$ is in $\Gamma$ or not, we find that, for the line bundle $\CL$:
\[\deg \CL = \begin{cases}
  g - 1 + \frac{\epsilon_\infty}{2} & -1 \not \in \Gamma  \\
  2g - 2 + \epsilon_\infty & -1 \in \Gamma  \end{cases}\]
where $\epsilon_\infty$ is the number of (regular) cusps. \\
When are these conditions (no elliptic points or irregular cusps) fulfilled for $\Gamma(N)$, $\Gamma_0(N)$, and $\Gamma_1(N)$? It is known that for these groups, the only example of an irregular cusp occurs for $\Gamma_1(4)$ (Section 3.8 of \cite{DS05}). Furthermore, the groups $\Gamma(N)$ for $N \geq 3$, and $\Gamma_1(N)$ for $N \geq 4$, have no elliptic elements (Exercise 2.3.7 of \cite{DS05}). 
\begin{lem}\label{GZell}The group $\Gamma_0(N)$ has no elliptic elements if and only if:
\[N \equiv 0 \pmod{4} \indent \mbox{ or } \indent N \equiv 0 \pmod{p}, p \equiv 3 \pmod{4}\]
and:
\[N \equiv 0 \pmod{9} \indent \mbox{ or } \indent N \equiv 0 \pmod{p}, p \equiv 5 \pmod{6}.\]
\end{lem}
\begin{proof}
This follows immediately from Corollary 3.7.2 in \cite{DS05}.
\end{proof}
We get, for example, that for a prime $p \geq 5$, $\Gamma_0(p)$ satisfies the necessary conditions if and only if $p \equiv 11 \pmod{12}$. \\
Using Lemma \ref{surjmul} and the above remarks, we can decided the maximal weight needed to generate the $\CC$-algebra $M(\Gamma)$ for many groups $\Gamma$. We have the following corollary:
\begin{cor}\label{cgen}
The $\CC$-algebra $M(\Gamma)$ of modular forms is generated:
\begin{enumerate}
\item in weight at most $3$, when $\Gamma = \Gamma_1(N)$ for $N \geq 5$, or
\item in weight at most $6$, when $\Gamma = \Gamma_0(N)$ for $N$ satisfying the conditions of Lemma \ref{GZell}.
\end{enumerate}
\end{cor}
\begin{proof}
We proceed as in the start of the proof of Theorem 5.1 in \cite{Makdisi12}. Let $\CL$ be the line bundle described above, whose global sections correspond to the modular forms in question. For (1), we note that in both cases, we have $-1 \not \in \Gamma$, and $\epsilon_\infty \geq 4$ (Section 3.8 of \cite{DS05}). By Lemma \ref{surjmul}, we seek a positive integer $k$ such that $k(g-1 + \frac{\epsilon_\infty}{2}) \geq 2g+1$. We see then that $k \geq 2$ works. Since $\gcd(2,3)=1$, we conclude that any subalgebra of $M(\Gamma)$ containing the modular forms of weights $2$ and $3$ must be the whole of $M(\Gamma)$, proving (1). \\
For (2), we note that when $N > 1$, we have $\epsilon_\infty \geq 2$. Thus we seek a positive integer $k$ such that $2kg \geq 2g+1$, and we see that $k\geq 2$ works, as long as $g \geq 1$. In that case, similarly as above, any subalgebra of $M(\Gamma)$ containing the modular forms of weights $4$ and $6$ is equal to the whole of $M(\Gamma)$. For the cases where $g=0$, this follows from the results of \cite{TS11}.
\end{proof}
We end with the following lemma that will be used later.
\begin{lem}\label{qgen}
Let $\Gamma$ be a congruence subgroup. The algebras $M(\Gamma, \QQ)$ and $M(\Gamma, \CC)$ are generated in the same weight. 
\end{lem}
\begin{proof}
This follows directly from the existence of $\ZZ$ basis for each $M_k(\Gamma,\CC)$, and comparing dimensions (\cite{Shi71}, Chapter 3).
\end{proof}

\section{Moduli schemes and modular forms}

Throughout this section, whenever we write ``modular form of level $N$", we are referring only to modular forms over the congruence subgroup $\GI(N)$. That is, we write $M_k(N,R)$ for $M_k(\GI(N),R)$, and $M(N,R)$ for $M(\GI(N),R)$.\\
To deal with modular forms over rings other than $\CC$, we need to generalize the notion of a modular curve to that of a moduli scheme. It turns out that for nice congruence subgroups, modular forms can be viewed as sections of an invertible sheaf on a projective moduli scheme, over a suitable base ring. We can use this description to extend the result of the previous section to $\ZZ[1/N]$-algebras of modular forms of level $N$.

\subsection{The moduli problem $\GI$}

The exposition in this section follows \cite{G90}. Let $S$ be an arbitrary scheme. By an elliptic curve $E/S$, we mean a proper smooth curve $\pi: E \rightarrow S$, whose geometric fibers are connected curves of genus 1, together with a section $0: S \rightarrow E$.
\[\begin{tikzpicture}[xscale=3,yscale=-1.5]
 \node (E) at (0,0) {$E$};
 \node (S) at (0,1) {$S$};
 \draw [->] (E) -- node [left]{$\pi$}(S);
 \draw [->] (S.north east) .. controls(0.15,0.5) .. node[right]{$0$}(E.south east);
\end{tikzpicture}\]
Given an elliptic curve $\pi : E \rightarrow S$, we have an invertible sheaf on $S$ given by:
\[ \omega_E = \pi_* \Omega^1_{E/S} \]
and whose formation commutes with base change. 
By a generalized elliptic curve, we mean a family of genus 1 curves whose fibers are either elliptic curves, or N\'eron polygons, together with a morphism $+ : E^{reg} \times_S E \rightarrow E$ whose restriction to $E^{reg}$ makes $E^{reg}$ into a commutative group scheme on $E$, and that on the fibres $E_s$ with singular points, the translations by $E^{reg}_s$ act by rotations on the graph of irreducible components. For a generalized elliptic curve $E/S$, we can define the invertible sheaf $\omega_E$ on $S$ as the dual of the sheaf of Lie algebras $Lie(E^{reg})$.\\
Let $E/S$ be an elliptic curve. A $\GI(N)$-structure on $E/S$, also called a ``point of exact order N" in E(S), is a homomorphism:
\[\alpha : \ZZ/N\ZZ \rightarrow E[N](S). \]
The point $P = \alpha(1)$ is the corresponding point of exact order $N$. \\
A moduli problem is a contravariant functor $\CP: SCH \rightarrow SETS$ from the category of schemes to that of sets. We consider the moduli problem $\CP_1$ classifying isomorphism classes of pairs $(E,\alpha)$ of generalized elliptic curves $E$ together with a $\GI(N)$-structure  $\alpha$ (or a point of exact order $N$). \\
We have the following key theorem (\cite{G90}):
\begin{thm}
Let $N \geq 5$. Let $\CP_1$ be the functor which assigns to each $\ZZ[1/N]$-scheme $S$ the set of isomorphism classes $[E, \alpha]$ consisting of a generalized elliptic curve $E/S$ and a $\GI(N)$-structure $\alpha$ on $E/S$. Then $\CP_1$ is representable by a smooth, proper, and geometrically connected algebraic curve $X_1(N)$ over $Spec(\ZZ[1/N])$. 
\end{thm}

\begin{rem}
One might consider instead the moduli functor $\CP$ classifying only elliptic curves with $\GI(N)$-structure. It turns out that for $N \geq 4$, this functor is representable by an affine scheme $Y_1(N)$, which can be thought of as a subscheme of $X_1(N)$. Thus $X_1(N)$ can be seen as the the compactification of $Y_1(N)$, obtained by adding the cusps, which correspond to generalized elliptic curves. See \cite{DI95}, Section 9.
\end{rem}

For $N \geq 5$, we denote by $\EE_1 / X_1$ the universal elliptic curve over $X_1(N)$, and we let $\som$ be the invertible sheaf on $X_1(N)$ as defined above. For a $\ZZ[\frac{1}{N}]$-algebra $R$, we write $X_1(N)_R$ for the moduli scheme obtained from $X_1(N)$ through base change, and we write $\som_R$ for the corresponding sheaf.

\subsection{Modular forms}
For our purposes, it is convenient to define modular forms in the following manner, again following \cite{G90}. 
\begin{defn}
Let $N \geq 5$, and $R$ a $\ZZ[1/N]$-algebra. A (holomorphic) modular form for $\GI(N)$, defined over $R$ and of weight $k$ is a global section of the sheaf $\som^{\otimes k}_R$. We write:
\[ M_k(N,R) = H^0 (X_1(N)_R, \underline{\omega}^{\otimes k}_R )\]
for the $R$-algebra of modular forms over $R$, with level $N$ and weight $k$. These modular forms generate a graded ring:
\[ M(N,R) = \bigoplus_{k=0}^{\infty} M_k(N,R).\]
\end{defn}
We have to check that for a $\ZZ[1/N]$-subalgebra $R$ of $\CC$, we recover the classical definition of modular forms with coefficients in $R$. For that, we need first to define the Fourier expansion of a modular form in an algebraic manner. This can be done using the Tate curve, which is a generalized elliptic curve $E_{Tate} = \GGm/q^{\ZZ}$ over $\ZZ[[q]]$. This curve has a canonical differential $dt/t$ and a natural embedding $Id_N: \mu_N \rightarrow E_{Tate}[N]$ over $Z[1/N][[q]]$. The Fourier expansion of $f$ is then defined to be $f(q)$ in the following identity:
\[f(\GGm/q^{\ZZ},Id_N) = f(q) (dt/t)^{\otimes k}.\]
There is a unique morphism $Spec(\ZZ[1/N][[q]]) \rightarrow X_1(N)$ such that the Tate curve arises as pull-back of the universal curve $\EE_1/X_1(N)$. The image of the prime ideal where $q=0$ defines the section $\infty$ of $X_1(N)$ and $q$ is a uniformizing parameter. Thus the Fourier expansion of $f$ is the holomorphic section $f$ of $\som^{\otimes k}$ near $\infty$. \\
When $R = \CC$, it is proven in \cite{DR72}, VII, \S{4}, that $f(q)$ will be the $q$-expansion of $f$ at $\infty$ in the classical sense. We also have the following theorem:
\begin{thm} (The $q$-expansion principle). Let $R$ be a $\ZZ[1/N]$-algebra.
\begin{enumerate}
\item The map $H^0(X_1(N)_R, \som^{\otimes k}_R) \rightarrow R[[q]]$ taking $f$ to $f(q)$ is an injection of $R$-modules.
\item if $R_0$ is a $\ZZ[1/N]$-subalgebra of $R$, the modular form $f$ is defined over $R_0$ if and only if $f(q) \in R_0[[q]]$. 
\end{enumerate}
\end{thm}
Applying this theorem with $R = \CC$, we see that for subrings $R_0$ of $\CC$ (in which $N$ is invertible), we recover the classical notion of modular forms whose $q$-expansion has coefficients in $R_0$. \\
In order to deal with modular forms in positive characteristic, we need the following base change theorem (\cite{DI95}, Theorem 12.3.2):
\begin{thm} If $B$ is an $A$-algebra and either one of the following holds:
\begin{enumerate} 
\item $B$ is flat over $A$, or
\item $k > 1$ and $N$ is invertible in $B$,
\end{enumerate}
then the natural map:
\[ M_k(N,A) \otimes_A B \rightarrow M_k(N,B) \]
is an isomorphism.
\end{thm}
In particular, we find that, when $p \nmid N$ and $k \geq 2$, we have $M_k(N,\FF_p) = M_k(N,\ZZ[1/N])\otimes_{\ZZ[1/N]} \FF_p$.

\subsection{Modular forms of level $N$ over $\ZZ[1/N]$}

\begin{prop}\label{fpgen}
Let $N \geq 5$, and let $p \nmid N$. Then the $\FF_p$-algebra $M(N,\FF_p)$ is generated in weight at most $3$. 
\end{prop}
\begin{proof}
First, we consider the situation over the algebraic closure $\FFc_p$. Denote respectively by $X_1(N)_{\FF_p}$ and $X_1(N)_{\FFc_p}$ the base change of $X_1(N)/\ZZ[1/N]$ to $\FF_p$ and $\FFc_p$ (note that these fields are $\ZZ[1/N]$-algebras since $p \nmid N$). The base change to $\FF_p$ corresponds to good reduction by Igusa's theorem (\cite{DS05}, Theorem 8.6.1), hence the genus of $X_1(N)_{\FFc_p}$ (which is equal to the genus of $X_1(N)_{\FF_p}$) is equal to $g$, the genus of the modular curve $X_1(N)_{\CC}$ (since the latter is equal to the genus of $X_1(N)_{\QQ}$ by flatness of base change). \\
For $k \geq 2$, We write $M_k(N,\FFc_p) = H^0 (X_1(N)_{\FFc_p}, \underline{\omega}_{\FFc_p}^{\otimes k})$. We want to find the degree of the invertible sheaf $\underline{\omega}_{\FFc_p}$. A direct way is to note that $\underline{\omega}^{\otimes(p-1)}_{\FFc_p}$ contains a special global section, which is the Hasse invariant $A \equiv E_{p-1} \pmod{p}$. It has zeroes precisely at the points of $X_1(N)_{\FFc_p}$ corresponding to isomorphism classes $[E,\alpha]$ where $E/\FFc_p$ is a supersingular elliptic curve. Thus we need to count the number of points on $E/\FFc_p$ corresponding to supersingular elliptic curves.\\
For an elliptic curve $E/\FFc_p$, $\CP_1(E)$ is the set of points of exact order $N$ on $E$. Since $E[n]$ has order $n^2$, by inclusion-exclusion we get that:
\[|\CP_1(E)| = N^2 \prod_{p | N} \left( 1-\frac{1}{p^2} \right) = 2[SL_2(\ZZ):\GI(N)]. \]
Let $r = |\CP_1(E)|$ and $P_1,\cdots,P_r$ the points of exact order $N$ on $E$. We want to count the number of distinct isomorphism classes in the set $\CP = \{[E,P_1],\cdots,[E,P_r]\}$. The group $Aut(E)$ acts on $\CP$, and by representability of the moduli functor (we are assuming $N \geq 5$), this action is free. Thus the number of orbits is $|\CP_1(E)|/|Aut(E)|$. Summing over supersingular curves, we get:
\[ \deg(\underline{\omega}_{\FFc_p}) = \frac{2[SL_2(\ZZ):\GI(N)]}{p-1} \sum_{E \mbox{ supersingular}} \frac{1}{|Aut(E)|} = \frac{1}{12} [SL_2(\ZZ):\GI(N)] \]
where for the last inequality we used the Eichler-Deuring mass formula:
\[  \sum_{E/\FFc_p \mbox{ supersingular}} \frac{1}{|Aut(E)|} = \frac{p-1}{24}. \]
The exact same numbers hold for the genus of $X_1(N)_{\FF_p}$ and the degree of the invertible sheaf $\underline{\omega}_{\FF_p}$, since base change along a field extension preserves the degrees of line bundles (\cite{BG06}). Thus we can apply Lemma \ref{surjmul} (surjectivity of multiplication) with exactly the same numbers as in the proof of Corollary \ref{cgen}, which proves the statement.
\end{proof}
We need the following lemma:
\begin{lem}\label{linalg}
Let $N,d$ be a positive integers, $V$ be a $\QQ$ vector space, and $v_1,\cdots,v_d \in V$. Let $v \in \spn_{\QQ}\{v_1,\cdots,v_d\}$. For a prime $p$, let $V_p = V \otimes_{\QQ} \QQ_p$. If for each prime $p \nmid N$ we have $v \in \spn_{\ZZ_p}\{v_1,\cdots,v_d\} \subseteq V_p$, then $v \in \spn_{\ZZ[1/N]}\{v_1,\cdots,v_d\}$. 
\end{lem}
\begin{proof}
Let $L = \spn_{\ZZ} \{v_1,\cdots,v_d\}$. The lattice $L$ is a finitely generated $\ZZ$-module inside a $\QQ$-vector space, so it is free, and has a $\ZZ$-basis $\{w_1,\cdots,w_e\}$ for some $e \leq d$. The vectors $w_1,\cdots,w_e$ are $\QQ$-linearly independent, thus can be extended to a $\QQ$-basis $\{w_1,\cdots,w_e,\cdots,w_t\}$ of $V$. Thus we can write $v = \sum_{i=1}^t a_i w_i$, for some $a_i \in \QQ$. \\
Note that for each prime $p$, we have $\spn_{\ZZ_p} \{w_1,\cdots,w_e\} = \spn_{\ZZ_p} \{v_1,\cdots,v_d\}$. Thus we have by the assumption that, for each prime $p \nmid N$, $v \in \spn_{\ZZ_p} \{w_1,\cdots,w_e\}$. Since $\{w_1,\cdots, w_t\}$ are also linearly independent in $\QQ_p$, this means that $a_i \in \QQ \cap \bigcap_{p \nmid N} \ZZ_p = \ZZ[1/N]$. As each $v_i$ is a $\ZZ$-linear combination of the vectors $\{w_1,\cdots,w_e\}$, the statement is proven.
\end{proof}

We can now prove Theorem 1:\\\\
\textbf{Theorem 1.} \textit{Let $N \geq 5$. The algebra $M(N,\ZZ[1/N])$ is generated in weight at most 3.}
\begin{proof}
Let $k \geq 4$, and $f \in M_k(N,\ZZ[1/N])$, and suppose for the purpose of contradiction that $f$ is not a polynomial in forms of weight less than $4$. By Lemma \ref{qgen} and Theorem \ref{cgen}, the $\QQ$-algebra $M(N,\QQ)$ is generated in weight at most $3$ by forms with coefficients in $\ZZ$, so there are monomials $g_1,\cdots,g_r \in M_k$ in these generators that have coefficients in $\ZZ$ and that span the $\QQ$-vector space $M_k(N,\QQ)$. \\
Denote by $f_0$ the reduction mod $p$ of $f$, i.e. $f_0 \equiv f \pmod{p}$. By Proposition \ref{fpgen}, we can find $a^{(0)}_1,\cdots,a^{(0)}_r \in \ZZ$ such that $f_0 \equiv \sum a^{(0)}_i g_i \pmod{p}$. This means, by our supposition, that $f_0 - \sum_{i=1}^r a^{(0)}_i g_i = pf_1$, where $f_1 \in M_k(N,R)$. We again have $a^{(1)}_1,\cdots,a^{(1)}_r \in \ZZ$ such that $f_1 \equiv \sum a^{(1)}_i g_i \pmod{p}$, that is, again by the supposition, $f_1 - \sum_{i=1}^r a^{(1)}_i g_i = pf_2$, where $f_2 \in M_k(N,R)$. Continuing in this manner, we may write $f = \sum_{i=1}^r a_i g_i$, where $a_i \in \ZZ_p$, i.e. $f \in \spn_{\ZZ_p}\{g_1,\cdots,g_r\}$. Since this holds for each prime $p \nmid N$, it follows by Lemma \ref{linalg} that $f \in \spn_{\ZZ[1/N]}\{g_1,\cdots,g_r\}$, as we wanted to show. 

\end{proof}

\section{Modular forms over $\GZ(N)$}
The method described in the previous section fails to carry over to modular forms over $\GZ(N)$. While we can define a moduli problem of type $\GZ(N)$, classifying elliptic curves with a specified subgroup of order $N$, i.e. pairs $(E,C)$, the corresponding moduli space is only a coarse moduli scheme and never a fine one, since for each $N$ the pair $(E,C)$ has at least one non-trivial automorphism (involution, corresponding to the element $-1 \in \GZ(N)$).\\
Throughout this section, whenever we write ``modular form of level $N$", we are referring only to modular forms over the congruence subgroup $\GZ(N)$, with coefficients in $\CC$. That is, we write $M_k(N)$ for $M_k(\GZ(N))$, $S_k(N)$ for the cuspforms $S_k(\GZ(N))$, and $M(N)$ for $M(\GZ(N))$. We first show that the algebra $M(\GZ(N))$ is generated in weight at most $10$ whenever $N$ is square-free. Then, we describe some work of Scholl related to the finite generation of these rings, leading to a rudimentary algorithm that calculates the generators. This allows us to present some experimental data, and suggests two conjectures, to be stated later. 

\subsection{The algebra $M(N)$ for $N$ square-free}
This section uses mainly the work of B\"ocherer and Nebe (\cite{BN10}). See also \cite{BA03}. Throughout this section, we work with modular forms on the congruence subgroup $\GZ(N)$, for $N$ square-free. We begin by the following definition:
\begin{defn}
Let $\mathcal{S}$ be a subspace of $S_k(N)$, of dimension $d$. The Weierstrass subspace of $\mathcal{S}$ is:
\[ \mathcal{WS} = \{f \in \mathcal{S}: \nu_{\infty}(f) >  d\}. \]
We say that $\mathcal{S}$ has the Weierstrass property if $\mathcal{WS} = \{ 0 \}$. Equivalently, this means that the projection:
\[ \sum a_n q^n \mapsto (a_1,\cdots,a_{d}) \] 
is injective. 
\end{defn}
Let $N$ be a square-free positive integer. Let $p|N$ be a prime. Pick an element $\omega_p \in SL_2(\ZZ)$ satisfying:
\[\omega_p \equiv \begin{pmatrix} 0 && -1 \\ 1 && 0 \end{pmatrix} \pmod{p}\]
and
\[\omega_p \equiv \begin{pmatrix} 1 && 0 \\ 0 && 1 \end{pmatrix} \pmod{\frac{N}{p}}.\]
The Atkin-Lehner involution $W^N_p$ then acts on $f \in M_k(N)$ by sending it to $f|_{W^N_p}$. We also have the operator:
\[ U_p(\sum a_n q^n) = \sum a_{np} q^n \]
acting on $M_k(N)$. Define the spaces: 
\[ M_k(N)^* = \{f \in M_k(N) : f|_{W^N_p} + p^{1-k/2} f|_{U(p)} = 0 \indent \forall \mbox{prime } p, p|N \}\]
and:
\[ S_k(N)^* = S_k(N) \cap M_k(N)^* \]
As B\"ocherer and Nebe note in the paper, $S_k(N)^*$ contains the space $S_k(N)^{\mbox{new}}$ of newforms. The key theorem here is proven in \cite{BA03}:
\begin{thm}
For $N$ square-free, and $k$ a positive integer, the space $S_{2k}(N)^*$ has the Weierstrass property.
\end{thm}
Define the graded ring:
\[ M(N)^* = \bigoplus_{k=0}^{\infty} M_k(N)^* \]
which is also a graded $M(1) = \CC[E_4, E_6]$ module. The dimension formula is given in Section 5.4 of \cite{BN10}:
\[ d_k(N) := \dim M_k(N)^* = \frac{(k-1)N}{12} + \frac{1}{2} - \frac{1}{4} \left( \frac{-1}{(k-1)N}\right) - \frac{1}{3}\left(\frac{-3}{(k-1)N}\right)\]
(the two last terms involve the Jacobi symbol). By computation of the Hilbert series for $M(N)^*$, we have (Theorem 5.12 of \cite{BN10}):
\begin{thm}\label{weierspace}
$M(N)^*$ is a free $\CC[E_4,E_6]$ module of rank $N$ generated by homogeneous elements of weight $\leq 10$. 
\end{thm}
This result allows us to prove that for square-free $N$, $M(N)$ is generated in weight at most 10. Recall the following standard decomposition result:
\[ M_k(N) = \mathcal{E}_k \oplus S_k(N)^{old} \oplus S_k(N)^{new}. \]
\begin{cor}\label{BNcor}
For square-free $N$, the ring $M(N)$ is generated in weight at most $10$.
\end{cor}
\begin{proof}
We use the decomposition of $M_k(N)$ into Eisenstein space, old space and new space, described above. First, we deal with the Eisenstein component. Let $Q_i$ be the cusps of $\GZ(N)$, and $\nu$ the number of cusps. For $k = 4$ or $6$, and for each cusp $Q_i$ of $\GZ(N)$, there is an Eisenstein series $f_{k,i}$ associated to $Q_i$ (this can be deduced from the similar result for modular forms over $\Gamma(N)$, which can be found in \cite{DS05}, Section 4.2; for a more general result, see \cite{Pet82}, Theorem G.1), in the sense that $f_{k,i}$ takes the value $1$ at $Q_i$ and vanishes at all other cusps. Let $k > 6$ be an even integer, and $F \in \mathcal{E}_k$. There exist positive integers $u,v$ such that $k = 4u + 6v$. Then for each cusp $Q_i$, the form $f^u_{4,i}f^v_{6,i}$ has weight $k$ and takes the value $1$ at $Q_i$ and vanishes at every other cusp. Thus we can find a suitable linear combination $G = \sum^{\nu}_{i=1} c_i f^u_{4,i}f^v_{6,i}$ such that $F - G$ is a cusp form. Therefore, Eisenstein series are generated by Eisenstein series in weights $4$ and $6$ in addition to cusp forms, and so we only need to prove that cusp forms are generated in weight at most $10$. \\
For the old part of the space of cusp forms, we use induction. Forms from level 1 are definitely generated in weight at most 10. Assuming the statement holds for levels $< N$, it will also hold for the old component of the space of cuspforms in level $N$. We then only have to prove the statement for the new component. As remarked above, $S_k(N)^*$ contains the new space, and so using Theorem \ref{weierspace}, the statement is established. 
\end{proof}
\subsection{Scholl's work}
For various subrings $A$ of $\CC$, Scholl (\cite{Scholl79}) provides an easy proof that the algebra $M(\GZ(N),A)$ is finitely generated. We state the theorem: 
\begin{thm} Let $\Gamma$ be a subgroup of $SL_2(\ZZ)$ of finite index, such that $-1 \in \Gamma$, $A$ a subring of $\CC$. Assume the cusp at infinity has width $1$, and $q = e^{2i\pi z}$ is a uniformizing parameter. 
If the following conditions hold:
\begin{enumerate}
\item $M_k(\Gamma, \CC) = M_k (\Gamma, A) \otimes \CC$
\item For some $t > 0$, there exists $T \in M_t(\Gamma,A)$ such that, $T$ is non-zero away from the cusp at infinity, and its q-expansion is a unit in $A((q))$. 
\end{enumerate}
Then $M(\Gamma, A)$ is finitely generated as an $A$-algebra.
\end{thm}
\begin{proof} Let $r$ be the vanishing order of $T$ at infinity. Write:
\[T = q^r \sum_{n=0}^{\infty} a_i q^i, \indent a_0 \in A^*,\indent a_n \in A\indent n \geq 1\]
Let $m_k = \dim_{\CC} M_k(SL_2(\ZZ),\CC)$. For for $F \in M_k(\Gamma, A)$, there exists a form $G \in M_k(SL_2(\ZZ), A)$ such that $F-G$ has vanishing order at least $m_k$ (to see this, we can use Victor-Miller basis to construct $G$). Thus if $m_k \geq r$, the function $\frac{F-G}{T}$ is a modular form of weight $k-t$ with coefficients in $A$ (follows from the defining properties of $T$). We can then use induction on $k$. Fix $k$ for which $m_k \geq t$, then write:
\[ M_k(\Gamma,A) = M_k(SL_2(\ZZ), A) + T\cdot M_{k-t}(\Gamma,A). \]
This shows that $M(\Gamma,A)$ is generated by $T$ and the forms in $M(SL_2(\ZZ), A)$ and $\bigoplus_{m_k < r} M_{k}(\Gamma,A)$. 
\end{proof}
\begin{defn}
A modular form satisfying the condition $(2)$ above is called a $T$-form. For example, when $\Gamma = SL_2(\ZZ)$, and $A = \ZZ$, then the discriminant $\Delta$ is such a $T$-form. 
\end{defn}
This gives a recipe to prove that, for given $\Gamma$ and $A$, the algebra $M(\Gamma,A)$ is finitely generated: we only have to produce a $T$-form. Once we have a $T$-form, we also have an algorithm to calculate a set of generators for $M(\Gamma,A)$.
\begin{algo}\label{algo1}\indent \\
\begin{enumerate}
\item $r$ = vanishing order of $T$-form at infinity.
\item $GENERATORS$ = A-basis of $M_2(\Gamma,A)$  (we know an integral basis exists).
\item for each $k \in 2\ZZ, k > 2, m_k < r$:
    \begin{enumerate} \label{thecheck}
    \item $BASIS$ = A-basis of $M_k(\Gamma, A)$.
    \item $MONOMIALS$ = The isobaric monomials in the elements of GENERATORS of weight $k$.
    \item for each $b \in BASIS$: Can $b$ be written as an $A$-linear combination of elements $m \in MONOMIALS$? (can check by linear algebra using Smith normal form)
        \begin{enumerate}
        \item if YES: do nothing.
        \item if NO: append $b$ to $GENERATORS$.
        \end{enumerate}
    \item if $k \in \{4,6, \mbox{weight of } T\}$
        \begin{enumerate}
        \item for $f = E_4$ or $E_6$ or $T$ (if $k$ is respectively $4,6$, or weight of $T$):  Can $f$ be expressed in terms of $g \in GENERATORS$? (check same as in \ref{thecheck} for weights 4,6, or weight of $T$)
        \item if YES: do nothing.
        \item if NO: append $f$ to $GENERATORS$.
        \end{enumerate}
    \end{enumerate}
\end{enumerate}
\end{algo}
\begin{thm} Algorithm 1 outputs a minimal list of generators for $M(\Gamma,A)$. \end{thm}
\begin{proof} Each generator added for each weight is indispensable. Since we are adding generators while increasing the weight, we have a minimal list of generators. \end{proof}
\begin{rem} For Algorithm \ref{algo1} to run as fast as possible, we should choose the $T$-form with the least possible order of vanishing at infinity.
\end{rem}
\begin{rem} A trick can be used to make the algorithm run much faster. For the sake of giving a concrete example, let us suppose we are working with the congruence subgroup $\Gamma_0(p)$, for prime $p\geq 5$. At each iteration, say for weight $k$, we can perform the following check. We let $S$ be the set of the vanishing orders of all the isobaric monomials in the elements of GENERATORS of weight $k$. If $S$ contains all integers from 0 up to and including the vanishing order of the $T$-form, then by taking suitable linear combinations, we can always divide by $T$ to reduce the case to a lower weight. Since $M_2(\Gamma_0(p),\ZZ)$ is not empty, and always contains a modular form which is non-vanishing at infinity (for example, an Eisenstein series), it would follow that for any higher weight, the situation described above will still be true, we can again reduce the case to a weight lower than $k$, and hence we can halt the algorithm.
\end{rem}

\subsection{Scholl's construction of the T-form for $\GZ(N)$}
Scholl's construction of the $T$-form on $\GZ(N)$ rests on the following lemma (\cite{Newman59}):
\begin{lem}
Let $N$ be a positive integer. Consider the eta product:
\[ f = \prod_{0<d|N} \eta(dz)^{r(d)} \]
where:
\begin{enumerate}
\item $r(d) \in \ZZ$ and $\sum r(d) = 0$, 
\item $\prod d^{r(d)}$ is a rational square,
\item $f$ has integral order of vanishing (that is, in the local parameter) at every cusp of $\GZ(N)$.
\end{enumerate}
Then $f$ is a modular function on $\GZ(N)$. 
\end{lem}
Let $Q_i$ be the cusps of of $\GZ(N)$, and $\frac{r_i}{s_i}$ be representatives of $Q_i$, with $Q_1$ being the cusp at infinity. Let $t_i$ be the width of $Q_i$. We can choose $r_1 = 1, s_1 = N$ and $t_1 = 1$. The set $G_N$ of eta products satisfying the above conditions is a multiplicative free, finitely generated abelian group, of rank at most $\sigma(N)-1$, where $\sigma(N)$ is the number of divisors of $N$. Let $\nu$ be the number of cusps. By $(f)$ we denote the divisor of $f$. We will need the following lemma (\cite{Scholl79}):
\begin{lem}\label{deltdiv}
Let $d$ be a positive divisor of $N$, and $r/s$ representing a cusp of $\GZ(N)$. Then the order of vanishing of $\Delta(dz)$ at $r/s$ is $t\gcd(d,s)^2/d$. 
\end{lem}
Recalling that $\eta^{24} = \Delta$, this allows us to calculate the divisor of an eta product, such as the ones described above. We have then the following proposition:
\begin{prop}
Let $\{n_1,\cdots,n_\nu\}$ be integers such that $\sum_{i=1}^{\nu} n_i = 0$. Then there is a function $f \in G_N$ and an integer $m$ such that 
\[(f) = m \sum_{i=1}^{\nu} n_i Q_i \]
if and only if for all $i,j$ such that $1 \leq i,j \leq \nu$, we have: 
\[ s_i = s_j \Rightarrow \frac{n_i}{t_i}=\frac{n_j}{t_j} \]
\end{prop}
\begin{proof} We briefly sketch Scholl's proof, since it provides a way to construct the $T$-form we are after. 
Write $x(d) = r(d)/m$. We need:
\[ (f) = \sum_{i=1}^{\nu} \sum_{0<d|N} (r(d)t_i\gcd(d,s_i)^2/24)Q_i = m \sum_{i=1}^{\nu} n_i Q_i \]
which means that we need, for each $1 \leq i \leq \nu$:
\[\sum_{0<d|N} \gcd(d,s_i)^2 x(d)/d = 24n_i / t_i \]
Scholl proves that this system is consistent and so $f$ and $m$ can be found.
\end{proof}
Having constructed such $f$ and $m$ as in the lemma, we can find a $T$-form for $\GZ(N)$:
\[ T = \frac{\Delta^m}{f} \]
and it has divisor:
\[ (T) = m \left(\sum_{i=1}^{\nu} n_i \right) Q_1, \]
that is, it only vanishes at infinity.\\\\
For $N = p$ a prime, we know that $\GZ(p)$ has only two cusps. It is then easy to explicitly solve for the smallest $m$ that works:
\[ m = \frac{p-1}{\gcd(24p,p-1)}. \]
We may need to multiply $m$ by $2$ in order to ensure that the second condition holsd, i.e. that $p^{r(p)}$ is a square. We get:
\[ r(1) = -r(p) = \frac{24mp}{p-1} = \frac{24p}{\gcd(24p,p-1)}\]
and therefore:
\[ f(z) = \left(\frac{\eta(z)}{\eta(pz)}\right)^{\frac{24p}{\gcd(24p,p-1)}} \]
giving the $T$-form:
\[ T(z)= \left(\frac{\eta(pz)^p}{\eta(z)}\right)^{\frac{24}{\gcd(24p,p-1)}} \]
with vanishing order at infinity:
\[ r = \frac{p^2-1}{\gcd(24p,p-1)} \]
and weight:
\[ w = \frac{12(p-1)}{\gcd(24p,p-1)}.\]
 We remark that Scholl's construction is not optimal, in the sense that it does not give the $T$-form with the lowest vanishing order at infinity. In general, the function:
\[ T'(z) = \left(\frac{\eta(pz)^{p}}{\eta(z)}\right)^2 \]
satisfies the properties of a $T$-form and has a vanishing order of $\frac{p^2-1}{12}$ at infinity and weight $p-1$. Calculations suggest that this $T$-form is the optimal $T$-form for prime level $p$. To see that it is actually a modular form on $\GZ(p)$, one can use the transformation formulae for the eta function, which can be found in \cite{Koehler11}. We can in fact prove that the optimal $T$-form must have weight either $p-1$ (thus being a constant multiple of $T'$), or $(p-1)/2$. To prove this, we need the following facts:
\begin{lem}\label{cuspmin}
Let $r$ be a positive integer. If there exists a modular function $f$ on $X_0(p)$, whose divisor is the cuspidal divisor $(f) = r( (\infty) - (0) )$, (cuspidal means supported at the cusps), then the least possible value for $r$ is:
\[r_p = \begin{cases}
  \frac{p-1}{12} & \text{ $p \equiv 1 \pmod{12}$} \\
  \frac{p-1}{4} & \text{ $p \equiv 5 \pmod{12}$} \\
  \frac{p-1}{6} & \text{ $p \equiv 7 \pmod{12}$} \\
  \frac{p-1}{2} & \text{ $p \equiv 11 \pmod{12}$} \end{cases}\]
Moreover, the group of principal cuspidal divisors on $X_0(p)$ is cyclic and generated by $r_p((\infty)-(0))$.
\end{lem}
\begin{proof}
It is straightforward to see that this group is cyclic. The other statement follows directly from Theorem 5.3 in \cite{KS10}.
\end{proof}
\begin{lem}\label{rwlem}
Let $T$ be a $T$-form for $\GZ(p)$, of weight $w$ and vanishing order $r$ at infinity. Then $r = \frac{w(p+1)}{12}$.
\end{lem}
\begin{proof}
Since $T$ only vanishes at infinity, the lemma follows directly from the valence formula (Proposition 2.16 of \cite{Shi71}).
\end{proof}
\begin{prop}
Let $T$ be a $T$-form for $\GZ(p)$, of least possible weight $w$. Then either $w = p-1$ or $w = (p-1)/2$.
\end{prop}
\begin{proof}
Because of the existence of $T'$, we know that $w \leq p-1$. Let $r$ be the vanishing order of $T$ at infinity. Since the level is prime, hence square-free, we have the Atkin-Lehner involution $W_p$ that acts on $M_w (\GZ(p))$ by:
\[ W_p(f) = p^{(2-w)/2} f|_{\gamma_p} \]
where:
\[\gamma_p = \begin{pmatrix}0 && -1 \\ p && 0 \end{pmatrix}\]
and $f \in M_w (\GZ(p))$. The Atkin-Lehner involution swaps the cusps of $\GZ(p)$. Thus the image of $T$ under the $W_p$ is a modular form that only vanishes at $0$. This means that the divisor of the modular function $g = \frac{T}{W_p(T)}$ is cuspidal and is equal to $r( (\infty) - (0))$. Assume that $p \equiv 1 \pmod{12}$ (the other cases are handled similarly). By Lemma \ref{cuspmin}, $r = m \frac{p-1}{12}$ for some $m \in \ZZ$, $m > 0$. By Lemma \ref{rwlem}, it follows that $w = m\frac{p-1}{p+1}$. It's clear then that $m$ is the least possible positive integer such that $m\frac{p-1}{p+1} \in \ZZ$, i.e. $m = \frac{p+1}{2}$ or $m = p+1$. As $p > 2$, we have $m = \frac{p+1}{2}$, since in that case $\gcd(p-1,p+1)=2$. Therefore, $w$ is multiple of $\frac{p-1}{2}$, and so it has to be either $\frac{p-1}{2}$ or $p-1$. 
\end{proof}

\subsection{Some calculations}

The computations in this section were carried out using SAGE. The following table shows the maximal weight needed to generated the $\ZZ[\frac{1}{6N}]$-algebra $M(\GZ(N),\ZZ[\frac{1}{6N}])$, up to the weight given in the third column. For certain levels (level 1 and prime levels), the entry in the third column is blank, since for these levels we have used Algorithm \ref{algo1} to prove that the whole algebra $M(\GZ(N),\ZZ[\frac{1}{6N}])$ is generated in the given maximal weight. 
\[\begin{tabular}{ l | c | r }
  \hline
  Level N & generated in weight & up to weight\\
  \hline       
  1 & 6 & --\\
  2 & 4 & --\\
  3 & 6 & --\\
  4 & 2 & 30\\
  5 & 4 & --\\
  6 & 2 & 16\\
  7 & 6 & --\\
  8 & 2 & 16\\
  9 & 2 & 18\\
  10 & 4 & 12\\
  11 & 4 & --\\
  13 & 6 & --\\
  17 & 4 & --\\
  19 & 6 & --\\

  \hline  
\end{tabular}\]

The following table shows the maximal weight needed to generated the $\CC$-algebra $M(\GZ(N),\CC)$, for levels $N$ up to 150, up to the weight given in the third column. For most levels, the entry in the third column is blank, which means that the maximal weight given generates the whole algebra. For those levels appearing in \cite{TS11} (that is, 2, 3, 4, 5, 6, 7, 8, 9, 10, 12, 16, 18, and 25), the conclusion is established there. For square-free levels, Corollary \ref{BNcor} tells us that it is enough to check for generation up to weight 10, and we will have generators for the whole algebra. For levels $N$ satisfying the conditions of Lemma \ref{GZell} (that is, $\GZ(N)$ has no elliptic points), Corollary \ref{cgen} tells us that it's enough to check for generation up to weight $6$. This leaves the levels 49, 50, 147 and 125 undecided. \\
\[\begin{tabular}{ l | c | r }
  \hline
  Level N & generated in weight & up to weight\\
  \hline       
1   &   6   &  --\\
2   &   4   &  --\\
3   &   6   &  --\\
4   &   2   &  --\\
5   &   4   &  --\\
6   &   2   &  --\\
7   &   6   &  --\\
8   &   2   &  --\\
9   &   2   &  --\\
10   &   4   &  --\\
11   &   4   &  --\\
12   &   2   &  --\\
13   &   6   &  --\\
14   &   2   &  --\\
15   &   2   &  --\\
16   &   2   &  --\\
17   &   4   &  --\\
18   &   2   &  --\\
19   &   6   &  --\\
20   &   2   &  --\\
21   &   6   &  --\\
22   &   2   &  --\\
23   &   4   &  --\\
24   &   2   &  --\\
25   &   4   &  --\\
26   &   4   &  --\\
27   &   2   &  --\\
28   &   2   &  --\\
29   &   4   &  --\\
30   &   2   &  --\\
  \hline  
\end{tabular}\]

\[\begin{tabular}{ l | c | r }
  \hline
  Level N & generated in weight & up to weight\\
  \hline  
31   &   6   &  --\\
32   &   2   &  --\\
33   &   2   &  --\\
34   &   4   &  --\\
35   &   2   &  --\\
36   &   2   &  --\\
37   &   6   &  --\\
38   &   2   &  --\\
39   &   6   &  --\\
40   &   2   &  --\\
41   &   4   &  --\\
42   &   2   &  --\\
43   &   6   &  --\\
44   &   2   &  --\\
45   &   2   &  --\\
46   &   2   &  --\\
47   &   4   &  --\\
48   &   2   &  --\\
49   &   6   &  12\\  
50   &   4   &  12\\
51   &   2   &  --\\
52   &   2   &  --\\
53   &   4   &  --\\
54   &   2   &  --\\
55   &   2   &  --\\
56   &   2   &  --\\
57   &   6   &  --\\
58   &   4   &  --\\
59   &   4   &  --\\
60   &   2   &  --\\
61   &   6   &  --\\
62   &   2   &  --\\
63   &   2   &  --\\
64   &   2   &  --\\
65   &   4   &  --\\
66   &   2   &  --\\
67   &   6   &  --\\
68   &   2   &  --\\
69   &   2   &  --\\
70   &   2   &  --\\

  \hline  
\end{tabular}\]

\[\begin{tabular}{ l | c | r }
  \hline
  Level N & generated in weight & up to weight\\
  \hline 

71   &   4   &  --\\
72   &   2   &  --\\
73   &   6   &  --\\
74   &   4   &  --\\
75   &   2   &  --\\
76   &   2   &  --\\
77   &   2   &  --\\
78   &   2   &  --\\
79   &   6   &  --\\
80   &   2   &  --\\
81   &   2   &  --\\
82   &   4   &  --\\
83   &   4   &  --\\
84   &   2   &  --\\
85   &   4   &  --\\
86   &   2   &  --\\
87   &   2   &  --\\
88   &   2   &  --\\
89   &   4   &  --\\
90   &   2   &  --\\
91   &   6   &  --\\
92   &   2   &  --\\
93   &   6   &  --\\
94   &   2   &  --\\
95   &   2   &  --\\
96   &   2   &  --\\
97   &   6   &  --\\
98   &   2   &  --\\
99   &   2   &  --\\
100   &   2   &  --\\
101   &   4   &  --\\
102   &   2   &  --\\
103   &   6   &  --\\
104   &   2   &  --\\
105   &   2   &  --\\
106   &   4   &  --\\
107   &   4   &  --\\
108   &   2   &  --\\
109   &   6   &  --\\
110   &   2   &  --\\
  \hline  
\end{tabular}\]

\[\begin{tabular}{ l | c | r }
  \hline
  Level N & generated in weight & up to weight\\
  \hline  
111   &   6   &  --\\
112   &   2   &  --\\
113   &   4   &  --\\
114   &   2   &  --\\
115   &   2   &  --\\
116   &   2   &  --\\
117   &   2   &  --\\
118   &   2   &  --\\
119   &   2   &  --\\
120   &   2   &  --\\
121   &   2   &  --\\
122   &   4   &  --\\
123   &   2   &  --\\
124   &   2   &  --\\
125   &   4   &  10\\
126   &   2   &  --\\
127   &   6   &  --\\
128   &   2   &  --\\
129   &   6   &  --\\
130   &   4   &  --\\
131   &   4   &  --\\
132   &   2   &  --\\
133   &   6   &  --\\
134   &   2   &  --\\
135   &   2   &  --\\
136   &   2   &  --\\
137   &   4   &  --\\
138   &   2   &  --\\
139   &   6   &  --\\
140   &   2   &  --\\
141   &   2   &  --\\
142   &   2   &  --\\
143   &   2   &  --\\
144   &   2   &  --\\
145   &   4   &  --\\
146   &   4   &  --\\
147   &   6   &  8\\
148   &   2   &  --\\
149   &   4   &  --\\
150   &   2   &  --\\

  \hline  
\end{tabular}\]

Numerical evidence appears to be in support of the following conjectures.
\begin{conj}
For any $N\geq 1$, the $\CC$-algebra $M(\GZ(N),\CC)$ is generated in weight at most 6.
\end{conj}
\begin{conj}
For any $N\geq 1$, the $\ZZ[\frac{1}{6N}]$-algebra $M(\GZ(N),\ZZ[\frac{1}{6N}])$ is generated in weight at most 6.
\end{conj}
We remark that for any $N$ for which Conjecture 2 holds, Conjecture 1 holds as well. Algorithm \ref{algo1} is however inefficient to carry out more detailed computations for higher levels, and a new method for calculating the generators is needed. 

\bibliographystyle{amsalpha}
\bibliography{biblio}

Nadim Rustom, Department of Mathematical Sciences, University of Copenhagen, Universitetsparken 5, 2100 Copenhagen \O\space Denmark\\ \texttt{rustom@math.ku.dk}
\end{document}